\documentclass[11pt]{amsart}

\usepackage{amsmath}
\usepackage{amscd}
\usepackage{amssymb}
\usepackage{latexsym}
\usepackage{mathabx}

\usepackage[arrow,curve,matrix,tips,2cell]{xy}
  \SelectTips{eu}{10} \UseTips
  \UseAllTwocells

\usepackage{calrsfs}
\usepackage[T1]{fontenc} 
\usepackage{textcomp}
\usepackage{times}
 \usepackage[scaled=0.92]{helvet}

\usepackage{geometry}
\newtheorem{theorem}{Theorem}[section]
\newtheorem{lemma}[theorem]{Lemma}
\newtheorem{corollary}[theorem]{Corollary}
\newtheorem{proposition}[theorem]{Proposition}
\theoremstyle{definition}
\newtheorem{remark}[theorem]{Remark}
\newtheorem{definition}[theorem]{Definition}

\newcommand{\bgl}{\begin{equation}}
\newcommand{\egl}{\end{equation}}
\newcommand{\bgln}{\begin{eqnarray}} 
\newcommand{\egln}{\end{eqnarray}}
\newcommand{\bglnoz}{\begin{eqnarray*}} 
\newcommand{\eglnoz}{\end{eqnarray*}}
\newcommand{\btheo}{\begin{theorem}}
\newcommand{\etheo}{\end{theorem}}
\newcommand{\blemma}{\begin{lemma}}
\newcommand{\elemma}{\end{lemma}}
\newcommand{\bproof}{\begin{proof}}
\newcommand{\eproof}{\end{proof}}
\newcommand{\bdefin}{\begin{definition}}
\newcommand{\edefin}{\end{definition}}
\newcommand{\bprop}{\begin{proposition}}
\newcommand{\eprop}{\end{proposition}}
\newcommand{\bcor}{\begin{corollary}}
\newcommand{\ecor}{\end{corollary}}

\newcommand{\cL}{\mathcal L}
\newcommand{\cM}{\mathcal M}
\newcommand{\cO}{\mathcal O}
\newcommand{\cP}{\mathcal P}
\newcommand{\cX}{\mathcal X}

\def\Az{\mathbb{A}}
\def\Cz{\mathbb{C}}
\def\Kz{\mathbb{K}}
\def\Lz{\mathbb{L}}
\def\Mz{\mathbb{M}}
\def\Nz{\mathbb{N}}
\def\Qz{\mathbb{Q}}

\def\Tz{\mathbb{T}}

\newcommand{\mfm}{\mathfrak m}
\newcommand{\mfn}{\mathfrak n}
\newcommand{\mfp}{\mathfrak p}
\newcommand{\mfq}{\mathfrak q}

\newcommand{\ma}{\mapsto} 
\newcommand\onto{\twoheadrightarrow} 
\newcommand{\Hom}{{\rm Hom}\,}
\newcommand{\Aut}{{\rm Aut}\,}
\newcommand{\id}{{\rm id}}
\newcommand{\abs}[1]{\lvert#1\rvert} 
\newcommand{\defeq}{\mathrel{:=}} 
\newcommand{\dop}{\text{: }} 
\newcommand{\tors}{{\rm tors}}
\newcommand{\rec}{{\rm rec}}
\newcommand{\lge}{\left\{} 
\newcommand{\rge}{\right\}} 
\newcommand{\lru}{\left(} 
\newcommand{\rru}{\right)} 
\newcommand{\lsp}{\left\langle} 
\newcommand{\cXp}{\right\rangle} 
\newcommand{\rukl}[1]{\lru #1 \rru}
\newcommand{\gekl}[1]{\lge #1 \rge}
\newcommand{\spkl}[1]{\lsp #1 \cXp} 
\newcommand{\menge}[2]{\gekl{ #1 \dop #2 }}
\newcommand{\ab}{\mbox{{\tiny \textup{ab}}}}
\newcommand{\Gab}{G^{\ab}}
\newcommand{\Real}{{\rm Re}}
\newcommand{\hpsi}{\widecheck{\psi}}
\newcommand{\Hab}[1]{\widecheck{G}^{\mathrm{ab}}_{#1}}
\newcommand{\Gal}{\mathrm{Gal}}

\begin{document}

\date{\today\ (version 1.0)} 
\title{Reconstructing global fields from Dirichlet $L$-series}\thanks{Part of this work was done whilst the first and third author enjoyed the hospitality of the University of Warwick (special thanks to
Richard Sharp for making it possible).} 
\author[G.~Cornelissen]{Gunther Cornelissen}
\address{Mathematisch Instituut, Universiteit Utrecht, Postbus 80.010, 3508 TA Utrecht, Nederland}
\email{g.cornelissen@uu.nl, h.j.smit@uu.nl}
\author[B.~de Smit]{Bart de Smit}
\address{Mathematisch Instituut, Universiteit Leiden, Postbus 9512, 2300 RA Leiden, Nederland}
\email{desmit@math.leidenuniv.nl}
\author[X.~Li]{Xin Li}
\address{School of Mathematical Sciences, Queen Mary University of London, Mile End Road, London E1 4NS, United Kingdom}
\email{xin.li@qmul.ac.uk}
\author[M.~Marcolli]{Matilde Marcolli}
\address{Mathematics Department, Mail Code 253-37, Caltech, 1200 E.\ California Blvd.\ Pasadena, CA 91125, USA}
\email{matilde@caltech.edu}
\author[H.~Smit]{Harry Smit}


\begin{abstract} \noindent We prove that two global fields are isomorphic if and only if there is an isomorphism of groups of Dirichlet characters that preserves $L$-series. 
\end{abstract}

\maketitle

\section{Introduction}

As was discovered by Ga{\ss}mann in 1928 \cite{Gassmann}, number fields are not uniquely determined up to isomorphism by their zeta functions. A theorem of Tate \cite{Tate} implies the same for global function fields. At the other end of the spectrum, results of Neukirch and Uchida \cite{NeukirchInv}, \cite{U3} \cite{U} state that the absolute Galois group does uniquely determine a global field. The better understood abelianized Galois group again does \emph{not} determine the field up to isomorphism at all, as follows from the description of its character group by Kubota (\cite{Kubota}, compare \cite{Onabe}, \cite{AS}).

In this paper, we prove that two global fields $\Kz$ and $\Lz$ are isomorphic if and only if there exists an isomorphism of groups of Dirichlet characters $\hpsi: \ \Hab{\Kz} \cong \Hab{\Lz}$ that preserves $L$-series: $L_{\Kz}(\chi) = L_{\Lz}(\hpsi(\chi))$ for all  $\chi \in \Hab{\Kz}$. A more detailed series of equivalences can be found in the main theorem \ref{THM} below. To connect this theorem to the previous paragraph, observe that the existence of $\hpsi$ without equality of $L$-series is the same as $\Kz$ and $\Lz$ having abelianized Galois groups that are isomorphic as topological groups, and that for the trivial character $\chi_{\text{triv}}$, we have $L_{\Kz}(\chi_{\text{triv}})=\zeta_{\Kz}$, so that preserving $L$-series at $\chi_{\text{triv}}$ is the same as $\Kz$ and $\Lz$ having the same zeta function. 

In global function fields, it seems we need to use the stated hypothesis, and we explicitly construct the isomorphism of function fields via a map of kernels of reciprocity maps (much akin to the final step in Uchida's proof \cite{U}). For number fields, we do not need the full hypothesis: we can prove the stronger result that for every number field, there exists a character of any chosen order $>2$ for which the $L$-series does not equal any other Dirichlet $L$-series of any other field (see Theorem \ref{thm:Bart}). The method here, however, is not via the kernel of the reciprocity map, but rather via representation theory. 

We briefly indicate the relation of our work to previous results. The main result was first stated in a 2010 preprint \cite{CM} by two of the current authors, who discovered it through methods from mathematical physics (\cite{CM}, now split into two parts \cite{CM1} and \cite{CLMS}). For function fields, the main result was proven by one of the authors in \cite{Cor}, using dynamical systems, referring, however, to \cite{CM} for some auxiliary results, of which we present the first published proofs in the current paper. Sections \ref{section:Mz} and \ref{section:NF} contain an independent proof for the number field case that was found by the second author in 2011. The current paper not only presents full and simplified proofs, but also uses only classical methods from number theory (class field theory, Chebotarev, Gr\"unwald-Wang, and inverse Galois theory). In \cite{S}, reconstruction of field extensions of a fixed rational function field from relative $L$-series is treated from the point of view of the method in sections \ref{section:Mz} and \ref{section:NF} of this paper.

After some preliminaries in the first section, we state the main result in section \ref{mainsection}. The next few sections outline the proof of the various equivalences in the main theorem, using basic class field theory and work of Uchida and Hoshi. In the final sections, we deal with number fields, and use representation theory to prove some stronger results.

\section{Preliminaries}\label{section:prelim}

In this section, we set notation and introduce the main object of study. 

A monoid is a semigroup with identity element. If $R$ is a ring, we let $R^*$ denote its group of invertible elements.

Given a global field $\Kz$, we use the word \emph{prime} to denote a prime ideal if $\Kz$ is a number field, and to denote an irreducible effective divisor if $\Kz$ is a global function field. Let $\cP_{\Kz}$ be the set of primes of $\Kz$.  If $\mfp \in \cP_{\Kz}$, let $v_{\mfp}$ be the normalized (additive) valuation corresponding to $\mfp$, $\Kz_{\mfp}$ the local field at $\mfp$, and $\cO_{\mfp}$ its ring of integers. 
Let $\Az_{\Kz,f}$ be the finite adele ring of $\Kz$, $\widehat{\cO}_{\Kz}$ its ring of finite integral adeles and $\Az_{\Kz,f}^*$ the group of ideles (invertible finite adeles), all with their usual topology. Note that in the function field case, infinite places do not exist, so that in that case, $\Az_{\Kz,f} = \Az_{\Kz}$ is the adele ring, $\widehat{\cO}_{\Kz}$ is the ring of integral adeles and $\Az_{\Kz,f}^* = \Az_{\Kz}^*$ is the full group of ideles. If $\Kz$ is a number field, we denote by $\cO_{\Kz}$ its ring of integers. If $\Kz$ is a global function field, we denote by $q$ the cardinality of the constant field. 

Let $I_{\Kz}$ be the multiplicative monoid of non-zero integral ideals/effective divisors of our global field $\Kz$, so $I_{\Kz}$ is generated by $\cP_{\Kz}$.
We extend the valuation to ideals: if $\mfm \in I_{\Kz}$ and $\mfp \in \cP_{\Kz}$, we define $v_{\mfp}(\mfm) \in \mathbb{Z}_{\geq 0}$ by requiring $\mfm = \prod_{\mfp \in \cP_{\Kz}} \mfp^{v_{\mfp}(\mfm)}$. Let $N$ be the \emph{norm function} on the monoid $I_{\Kz}$: it is the multiplicative function defined on primes $\mfp$ by $N(\mfp):=\# \cO_{\Kz}/\mfp$ if $\Kz$ is a number field and $N(\mfp) = q^{\deg(\mfp)}$ if $\Kz$ is a function field. Given two global fields $\Kz$ and $\Lz$, we call a monoid homomorphism $\varphi: \ I_{\Kz} \to I_{\Lz}$ \emph{norm-preserving} if $N(\varphi(\mfm)) = N(\mfm)$ for all $\mfm \in I_{\Kz}$.

We have a canonical projection $$(\cdot)_{\Kz}: \ \Az_{\Kz,f}^* \cap \widehat{\cO}_{\Kz} \to I_{\Kz}, \, (x_{\mfp})_{\mfp} \ma \prod_{\mfp} \mfp^{v_{\mfp}(x_{\mfp})}.$$ A \emph{split for $(\cdot)_{\Kz}$} is by definition a monoid homomorphism $s_{\Kz}: \ I_{\Kz} \to \Az_{\Kz,f}^* \cap \widehat{\cO}_{\Kz}$ with the property that $(\cdot)_{\Kz} \circ s_{\Kz} = \id_{I_{\Kz}}$, and such that for every prime $\mfp$, $s_{\Kz}(\mfp) = (\ldots, 1, \pi_{\mfp}, 1, \ldots)$ for some uniformizer $\pi_{\mfp} \in \Kz_{\mfp}$ with $v_{\mfp}(\pi_{\mfp}) = 1$.

Let $\Gab_{\Kz}$ be the Galois group of a maximal abelian extension $\Kz^{\mathrm{ab}}$ of $\Kz$, a profinite topological group. There is an \emph{Artin reciprocity map} $\Az_{\Kz}^* \rightarrow \Gab_{\Kz}$. In the number field case, we embed $\Az_{\Kz,f}^*$ into the group of ideles $\Az_{\Kz}^*$ via $\Az_{\Kz,f}^* \ni x \ma (1,x) \in \Az_{\Kz}^*$, restrict the Artin reciprocity map to $\Az_{\Kz,f}^*$ and call this restriction $\rec_{\Kz}$. In the function field case, $\rec_{\Kz}$ is just the (full) Artin reciprocity map.

Let $\Tz$ denote the unit circle, equipped with the discrete topology, and $\Hab{\Kz}= \Hom (\Gab_{\Kz}, \Tz)$ the group of continuous linear characters of $\Gab_{\Kz}$. Given $\chi \in \Hab{\Kz}$,
we write $$U(\chi) \defeq \menge{\mfp \in \cP_{\Kz}}{\chi \vert_{\rec_{\Kz}(\cO_{\mfp}^*)} = 1}$$ for the set of \emph{primes where $\chi$ is unramified}. We denote by $\spkl{U(\chi)}$ the submonoid of $I_{\Kz}$ generated by $U(\chi)$, i.e., $\mfm \in I_{\Kz}$ is in $\spkl{U(\chi)}$ if and only if $v_{\mfp}(\mfm) = 0$ for all $\mfp \in \cP_{\Kz} \setminus U(\chi)$. For $\mfm \in \spkl{U(\chi)}$, we set (in a well-defined way, as the choice of $s_{\Kz}$ is up to an element of $\widehat{\cO}_{\Kz}^\ast$): 
\[
\chi(\mfm) \defeq \chi(\rec_{\Kz}(s_{\Kz}(\mfm))),
\]
and for $\mfm \in I_{\Kz} \backslash \spkl{U(\chi)}$ we set $\chi(\mfm) = 0$.

For any $\chi \in \Hab{\Kz}$, the kernel $\ker \chi$ is an open and closed subgroup of $\Gab_{\Kz}$. Therefore, the fixed field of $\chi$, denoted $\Kz_{\chi}$, is a finite Galois extension of $\Kz$. As the extension is finite, there is an $n$ such that $\chi^n$ is the trivial character. It follows that $\text{im} \,\chi$ is a subgroup of the $n^{\text{th}}$ roots of unity, hence cyclic. As $\text{im} \,\chi \cong \Gal(\Kz_{\chi}/\Kz)$, we obtain that $\Kz_{\chi} / \Kz$ is a finite cyclic extension. Conversely, for any finite cyclic extension $\Kz'$ of $\Kz$ there exists a character $\chi \in \Hab{\Kz}$ such that $\ker \chi = \Gal(\Kz^{\text{ab}}/\Kz')$. 

\blemma \label{lem:unram_primes}
Let $\chi \in \Hab{\Kz}$. The primes in $U(\chi)$ are exactly the primes that are unramified in $\Kz_{\chi}$. 
\elemma

\bproof
Using any section of the quotient map $\Gab_{\Kz} \twoheadrightarrow \Gal(\Kz_{\chi}/\Kz)$, the character $\chi$ induces an injective character $\overline{\chi}: \Gal(\Kz_{\chi}/\Kz) \to \Tz$. Under this quotient map, the set $\rec_{\Kz}(\cO_{\mfp}^\ast)$ is mapped surjectively to the inertia group $I_{\mfp}(\Kz_{\chi}/\Kz)$, hence we have $\chi(\rec_{\Kz}(\cO_{\mfp}^\ast)) = \overline{\chi}(I_{\mfp}(\Kz_{\chi}/\Kz))$. As $\overline{\chi}$ is injective, this set is equal to $\{1\}$ precisely when the inertia group is trivial, i.e. when $\mfp$ is unramified in $\Kz_{\chi}$.
\eproof

\blemma\label{lem:max_ext}
For any prime $\mfp \in \cP_{\Kz}$, set $N_{\mfp}:=\bigcap\limits_{\chi: \; \mfp \in U(\chi)} \ker \chi$. Then $\rec_{\Kz}(\cO_{\mfp}^*) = N_{\mfp}$, and the associated fixed field is equal to $\Kz^{\mathrm{ur},\mfp}$, the maximal abelian extension of $\Kz$ unramified at $\mfp$. 
\elemma

\bproof
By definition of $U(\chi)$, for any $\chi \in \Hab{\Kz}$ with $\mfp \in U(\chi)$ we have $\rec_{\Kz}(\cO_{\mfp}^*) \subseteq \ker \chi$, hence $\rec_{\Kz}(\cO_{\mfp}^*) \subseteq N_{\mfp}.$ As $\cO_{\mfp}^*$ is compact and $\rec_{\Kz}$ is continuous, $\rec_{\Kz}(\cO_{\mfp}^*)$ is compact, and as $\Gab_{\Kz}$ is Hausdorff, $\rec_{\Kz}(\cO_{\mfp}^*)$ is closed. 

Let $\Kz_{\cO_{\mfp}^*}$ denote the fixed field of $\rec_{\Kz}(\cO_{\mfp}^*)$. Under the quotient map $\Gab_{\Kz} \twoheadrightarrow \Gal(\Kz_{\cO_{\mfp}^*} / \Kz)$, $\rec_{\Kz}(\cO_{\mfp}^*)$ is mapped to the inertia group $I_{\mfp}(\Kz_{\cO_{\mfp}^*} / \Kz)$ by class field theory, but it is also mapped to $\{1\}$ by definition. Hence the inertia group is trivial, and 
$$(\Kz^{\mathrm{ab}})^{\rec_{\Kz}(\cO_{\mfp}^*)} = \Kz_{\cO_{\mfp}^*} \subseteq \Kz^{\mathrm{ur},\mfp}.$$

The extension associated to $N_{\mfp}$  
contains the composite of the extensions $\Kz_\chi$ associated to the individual $\chi$. The field $\Kz^{\mathrm{ur},\mfp}$ is a composite of finite cyclic extensions  of $\Kz$ unramified at $\mfp$. As mentioned, for any such a finite cyclic extension $\Kz'$ there exists a $\chi$ such that $\ker \chi = \Gal(\Kz^{\text{ab}}/\Kz')$. Because this extension is unramified at $\mfp$, the previous lemma asserts that $\mfp \in U(\chi)$. It follows that  $$\Kz^{\mathrm{ur},\mfp} \subseteq (\Kz^{\mathrm{ab}})^{N_{\mfp}}.$$

Combining these two results, we obtain $\rec_{\Kz}(\cO_{\mfp}^*) \supseteq N_{\mfp}$, and the result follows. 
\eproof

The \emph{(Dirichlet) $L$-function} attached to a character $\chi$ of $\Kz$ is given by the function of the complex variable $s$:
$$L_{\Kz}(\chi) = \prod_{\mfp \in U(\chi)} (1 - \chi(\mfp) N(\mfp)^{-s})^{-1} = \sum_{\mfm \in \spkl{U(\chi)}} \chi(\mfm) N(\mfm)^{-s}.$$
(We drop the field $\Kz$ from the notation for the $L$-series unless confusion might arise.)

Let $\chi$ be a character with associated extension $\Kz_\chi$. For the primes in $U(\chi)$ we have that $\rec_{\Kz}(s_{\Kz}(\mfp)) \mbox{ mod } \ker \chi$ is independent of the choice of the split $s_{\Kz}$, and equal to the \emph{Frobenius} $\mathrm{Frob}_\mfp$ in $\mathrm{Gal}(\Kz'/\Kz)$; note that in a general extension, ``the'' Frobenius of an unramified prime in $\Kz$ is a conjugacy class in the Galois group, but in our setting of abelian extensions, it is an actual element. 
The $L$-series can thus be written as
$$
L(\chi) 
  = \prod_{\mfp \in U(\chi)} (1 - \chi(\mathrm{Frob}_\mfp) N(\mfp)^{-s})^{-1}. 
$$
\section{The main theorem} \label{mainsection}

\btheo
\label{THM}
Let $\Kz$ and $\Lz$ be two global fields. The following are equivalent:
\begin{enumerate} 
\item[(i)] There exists a monoid isomorphism $\varphi: \ I_{\Kz} \cong I_{\Lz}$, an isomorphism of topological groups $\psi: \ \Gab_{\Kz} \cong \Gab_{\Lz}$ and splits $s_{\Kz}: \ I_{\Kz} \to \Az_{\Kz,f}^* \cap \widehat{\cO}_{\Kz}$, $s_{\Lz}: \ I_{\Lz} \to \Az_{\Lz,f}^* \cap \widehat{\cO}_{\Lz}$ such that
\bgln
\label{iv1}
  && \psi(\rec_{\Kz}(\cO_{\mfp}^*)) = \rec_{\Lz}(\cO_{\varphi(\mfp)}^*) \ \text{for every prime} \ \mfp \ \text{of} \ \Kz, \\
\label{iv2}
  && \psi(\rec_{\Kz}(s_{\Kz}(\mfm))) = \rec_{\Lz}(s_{\Lz}(\varphi(\mfm))) \ \text{for all} \ \mfm \in I_{\Kz}.
\egln
\item[(ii)] There exists 
\begin{itemize} 
\item a norm-preserving monoid isomorphism $\varphi: \ I_{\Kz} \cong I_{\Lz}$, and 
\item an isomorphism of topological groups $\psi: \ \Gab_{\Kz} \cong \Gab_{\Lz}$,
\end{itemize}
such that for every finite abelian extension $\Kz'=\left(\Kz^{\ab}\right)^N$ of $\Kz$ ($N$ a subgroup in $G_{\Kz}^{\ab}$) with corresponding field extension $\Lz' =\left(\Lz^{\ab}\right)^{\psi(N)}$ of $\Lz$, $\varphi$ is a bijection between the unramified primes of $\Kz'/\Kz$ and $\Lz'/\Lz$ such that
\[
\psi(\normalfont{\text{Frob}}_{\mfp}) = \normalfont{\text{Frob}}_{\varphi(\mfp)}.
\] 
\item[(iii)] There exists an isomorphism of topological groups $\psi: \ \Gab_{\Kz} \cong \Gab_{\Lz}$ such that
\bgl
\label{v}
  L(\chi) = L(\hpsi(\chi)) \ \text{for all} \ \chi \in \Hab{\Kz},
\egl
where $\hpsi$ is given by $\hpsi(\chi) = \chi \circ \psi^{-1}$.
\item[(iv)] $\Kz$ and $\Lz$ are isomorphic as fields.
\end{enumerate}
\etheo

We will refer to condition (i) as a \emph{reciprocity isomorphism}, condition (ii) as a \emph{finite reciprocity isomorphism}, and condition (iii) as an \emph{L-isomorphism}. We will prove the implication (i) $\Rightarrow$ (ii) in Proposition~\ref{equifin}, implication (ii) $\Rightarrow$ (iii) in Proposition~\ref{reci-->LIso}, and implication (iii) $\Rightarrow$ (i) in Proposition~\ref{LIso-->reci}. The equivalence with (iv) is proven in \ref{PROP} for function fields and follows from \ref{thm:Bart} for number fields. 
\section{From reciprocity isomorphism to finite reciprocity isomorphism}
\bprop
\label{reci-->N}
Assume \textup{\ref{THM}(i)}.
Then $\varphi$ is norm-preserving, i.e.
\bgl
\label{N=Nphi}
  N(\mfp) = N(\varphi(\mfp))\ \text{for all} \ \mfp \in \cP_{\Kz}.
\egl
\eprop
\bproof
We have $$\cO_{\mfp}^* \cong \rec_{\Kz}(\cO_{\mfp}^*) \cong \rec_{\Lz}(\cO_{\varphi(\mfp)}^*) \cong \cO_{\varphi(\mfp)}^*$$ for all $\mfp \in \cP_{\Kz}$ as the local reciprocity map is injective. Thus, it suffices to find a way to read off $N(\mfp)$ from the isomorphism type of $\cO_{\mfp}^*$. Let $\tors(A)$ be the torsion subgroup of an abelian group $A$, and let $A_p$ be the $p$-primary part of a finite abelian group $A$. Assume that $N(\mfp) = p^f$ for some prime $p$ and $f \geq 1$. The following facts follow from \cite[Chapter~II, Proposition~(5.7)]{Neu}:
\begin{itemize}
\item $p$ is uniquely determined by the property $\rukl{\cO_{\mfp}^* / \tors(\cO_{\mfp}^*)}^p \neq \cO_{\mfp}^* / \tors(\cO_{\mfp}^*)$,
\item $N(\mfp) = p^f = \abs{\tors(\cO_{\mfp}^*) / \tors(\cO_{\mfp}^*)_p} + 1$.
\end{itemize}
So we indeed have \eqref{N=Nphi}.
\eproof

\bprop \label{equifin}
Condition \textup{\ref{THM}(i)} implies condition \textup{\ref{THM}(ii)}.
\eprop
\bproof
A prime $\mfp$ of $\Kz$ is unramified in $\Kz'$ if and only if $\rec_{\Kz}(\cO^*_\mfp) \subseteq N$. Therefore, (1) implies that $\varphi(\mfp)$ is unramified in $\Lz'$, and the bijection follows by symmetry. Also, $\rec_{\Kz}(s_{\Kz}(\mfp)) \mbox{ mod } N$ is independent of the choice of the split $s_{\Kz}$, and equal to the Frobenius $\mathrm{Frob}_\mfp$ in $\mathrm{Gal}(\Kz'/\Kz)$. Hence from (2) we obtain $\psi(\mathrm{Frob}_\mfp) = \mathrm{Frob}_{\varphi(\mfp)}$. This proves 
\textup{(ii)}. 
\eproof

\section{From finite reciprocity isomorphism to L-isomorphism}

\bprop
\label{reci-->LIso}
Assume \textup{\ref{THM}(ii)}. Then $L(\chi) = L(\hpsi(\chi))$ for all $\chi \in \Hab{\Kz}$.
\eprop
\bproof Let $\chi \in \Hab{\Kz}$, and let $\Kz_{\chi} = (\Kz^{\rm ab})^{\ker \chi}$ denote the fixed field of $\ker \chi$. As mentioned in Section~\ref{section:prelim}, the associated $L$-series can be written as
\bglnoz
L(\chi) &=& \prod_{\mfp \in U(\chi)} (1 - \chi(\mathrm{Frob}_\mfp) N(\mfp)^{-s})^{-1}. 
\eglnoz 
We have $\psi (\ker \chi )= \ker ( \hpsi (\chi))$ by definition of $\hpsi$. Let $\Lz_{\hpsi(\chi)} = (\Lz^{\rm ab})^{\psi(\ker \chi)}$. Let $\mfp \in \cP_{\Kz}$ and $\mfq:= \varphi(\mfp) \in \cP_{\Lz}$. By assumption, $\varphi(U(\chi)) = U(\hpsi(\chi))$, $N(\mfp)=N(\mfq)$ and if $\mfp \in U(\chi)$ then $\psi(\mathrm{Frob}_\mfp)=\mathrm{Frob}_\mfq$ and thus $\chi(\mathrm{Frob}_\mfp) = \hat \psi(\chi)(\mathrm{Frob}_\mfq)$. It follows that
\begin{align*}
  L(\chi) 
  &= \prod_{\mfp \in U(\chi)} (1 - \chi(\mathrm{Frob}_\mfp) N(\mfp)^{-s})^{-1} \\
  &= \prod_{\mfq \in U(\hpsi(\chi))} (1 - \hpsi(\chi)(\mathrm{Frob}_\mfq) N(\mfq)^{-s})^{-1} \\
  &= L(\hpsi(\chi)). \qedhere
\end{align*} 
\eproof

\section{From L-isomorphism to reciprocity isomorphism} \label{LtoR}

\bprop
\label{LIso-->reci}
Assume that there exists an isomorphism of topological groups $\psi: \ \Gab_{\Kz} \cong \Gab_{\Lz}$ with $L(\chi) = L(\hpsi(\chi))$ for all $\chi \in \Hab{\Kz}$. Then \textup{\ref{THM}(i)} holds. 
\eprop

The idea of the proof is to construct a bijection of primes $\varphi$ with the desired properties via the use of a type of characters that have value $1$ on all but one of the primes of a certain norm.

For the proof, which will occupy this entire section, we will use the following notation. Whenever an object has subscript $N, <N$, or $\geq N$, all associated sets of (prime) ideals are restricted to (prime) ideals of norm $N, <N$ or $\geq N$ respectively. For example, $U_N(\chi)$ is the set of primes of norm $N$ at which $\chi$ is unramified. We use the multiplicative form of the $L$-series, i.e. 
\[
L_{<N}(\chi) = \prod_{\mfp \in U_{<N}(\chi)} (1 - \chi(\mathrm{Frob}_\mfp) N(\mfp)^{-s})^{-1}
\]
and $L_{\geq N}(\chi)$ is defined similarly.

Our approach is as follows: for every $N \in \Nz$ we construct a bijection $$\varphi_N: \: \cP_{\Kz, N} \to \cP_{\Lz, N}$$ such that $\chi(\mfp) = \hpsi(\chi)(\varphi_N(\mfp))$ for all $\chi \in \Hab{\Kz}$ and all $\mfp \in \cP_{\Kz, N}$. That suffices to prove (i), as we now first explain. We obtain a bijection $$\varphi: \: \cP_{\Kz} \to \cP_{\Lz}$$ by setting $\varphi \vert_{\cP_{\Kz, N}} \defeq \varphi_N,$ so that $\varphi$ satisfies $\chi(\mfp) = \hpsi(\chi)(\varphi(\mfp))$, for all $\chi \in \Hab{\Kz}$ and all $\mfp \in \cP_{\Kz}$. From this, the following sequence of equivalences follows:
\[
\mfp \in U(\chi) \Longleftrightarrow \chi(\mfp) \neq 0 \Longleftrightarrow \hpsi(\chi)(\varphi(\mfp)) \neq 0 \Longleftrightarrow \varphi(\mfp) \in U(\hpsi(\chi)).
\]
Because $\varphi$ is a bijection, we obtain $\varphi(U(\chi)) = U(\hpsi(\chi))$. Certainly $\varphi$ gives rise to a monoid isomorphism $I_{\Kz} \cong I_{\Lz}$, and \eqref{iv1} follows from Lemma~\ref{lem:max_ext} combined with $\varphi(U(\chi)) = U(\hpsi(\chi))$. 
Thus, for every $\mfp \in \cP_{\Kz}$, $\chi(\mfp) = \hpsi(\chi)(\varphi(\mfp))$ for all $\chi \in \Hab{\Kz}$ with $\mfp \in U(\chi)$ implies that for any splits $s_{\Kz}$ and $s_{\Lz}$, $$\psi(\rec_{\Kz}(s_{\Kz}(\mfp))) \equiv \rec_{\Lz}(s_{\Lz}(\varphi(\mfp))) \ \mathrm{mod} \ \rec_{\Lz}(\cO_{\varphi(\mfp)}^*).$$ Hence we can modify our splits to obtain \eqref{iv2}, and this proves (i).\\

To construct $\varphi_N$, the strategy is to proceed inductively on $N$. For $N=1$, there is nothing to do.
Assume that we have constructed $\varphi_M$ for all $M < N$. Let $$\varphi_{< N}: \cP_{\Kz, < N} \to \cP_{\Lz, < N}$$ be the norm-preserving bijection obtained by combining the $\varphi_M$.
Since $$U_{< N}(\hpsi(\chi)) = \varphi_{<N}(U_{<N}(\chi)), \ \chi(\mfp) = \hpsi(\chi)(\varphi_{< N}(\mfp)), \mbox{ and } N(\mfp) = N(\varphi_{< N}(\mfp)),$$ it follows that $$L_{< N}(\chi) = L_{< N}(\hpsi(\chi)).$$ Since $L(\chi) = L(\hpsi(\chi))$ and $L(\chi) = L_{< N}(\chi)L_{\geq N}(\chi)$ we also have $$L_{\geq N}(\chi) = L_{\geq N}(\hpsi(\chi)).$$ We can apply this equality to prove the following lemma.

\blemma \label{lem:sum_equal}
For any character $\chi \in \Hab{\Kz}$, define 
\[
\cX_N(\chi) := \sum_{\mfp \in \cP_{\Kz, N}} \chi(\mfp) = \sum_{\mfp \in U_N(\chi)} \chi(\mfp).
\]
Then $\cX_N(\chi) = \cX_N(\hpsi(\chi))$ for all $\chi \in \Hab{\Kz}$. As a result, $|\cP_{\Kz, N}| = |\cP_{\Lz, N}|$.
\elemma

\bproof
Let $\spkl{\cP_{\Kz, \geq N}}$ be the submonoid of $I_{\Kz}$ generated by $\cP_{\Kz, \geq N}$. We write $L_{\geq N}(\chi)$ in additive form:
\[
L_{\geq N}(\chi) = \sum_{\mfm \in \spkl{\cP_{\Kz, \geq N}}} \chi(\mfm) N(\mfm)^{-s} = \sum_{M \geq N} \left(\sum_{\mfm \in \spkl{\cP_{\Kz, \geq N}} \cap I_{\Kz,M}} \chi(\mfm) \right) M^{-s}.
\]
As $L_{\geq N}(\chi) = L_{\geq N}(\hpsi(\chi))$, the coefficients of $N^{-s}$ in both sums are equal, hence
\[
\sum_{\mfm \in \spkl{\cP_{\Kz, \geq N}} \cap I_{\Kz,N}} \chi(\mfm) = \sum_{\mfn \in \spkl{\cP_{\Lz, \geq N}} \cap I_{\Lz,N}} \hpsi(\chi)(\mfn).
\]
As $\spkl{\cP_{\Kz, \geq N}} \cap I_{\Kz,N} = \cP_{\Kz, N}$ and $\spkl{\cP_{\Lz, \geq N}} \cap I_{\Lz,N} = \cP_{\Lz, N}$, the desired equality follows. The final result is obtained by setting $\chi = 1$. 
\eproof

\bdefin
Let $c_N := |\cP_{\Kz, N}| = |\cP_{\Lz, N}|$. For any character $\chi \in \Hab{\Kz}$, define $u_N(\chi) = |U_N(\chi)|$, $V_N(\chi) = \{ \mfp \in U_N(\chi) \mid \chi(\mfp) = 1\}$, and $v_N(\chi) = |V_N(\chi)|$. If $v_N(\chi) = u_N(\chi) - 1$, there exists a unique prime of norm $N$ on which $\chi$ has a value of neither $0$ nor $1$. We will denote this prime by $\mfp_\chi$. 
Lastly, we define the following sets of characters:
\begin{align*}
\Xi_{\Kz}^1 &:= \left\{\chi \in \Hab{\Kz}: u_N(\chi) = v_N(\chi) = c_N \right\} \\
\Xi_{\Kz}^2 &:= \left\{\chi \in \Hab{\Kz}: u_N(\chi) = c_N,\; v_N(\chi) = c_N - 1, \; \chi(\mfp_\chi) = \zeta \right\},
\end{align*}
where $\zeta = \exp(2 \pi i / k)$ for some fixed integer $k \geq 3$.
\edefin

\begin{remark}
Let $\chi \in \Hab{\Kz}$. As $|\chi(\mfp)| = 1$ for all $\mfp \in U_N(\chi)$, we have $\Real(\cX_N(\chi)) \leq u_N(\chi)$. Equality holds precisely when $\chi(\mfp) = 1$ for all $\mfp \in U_N(\chi)$, i.e. $u_N(\chi) = v_N(\chi)$. 
\end{remark}

\blemma \label{Delta1}
$\hpsi(\Xi_{\Kz}^1) = \Xi_{\Lz}^1$. 
\elemma

\bproof
From Lemma~\ref{lem:sum_equal} we obtain $\chi \in \Xi_{\Kz}^1 \Longleftrightarrow \cX_N(\chi) = c_N \Longleftrightarrow \cX_N(\hpsi(\chi)) = c_N \Longleftrightarrow \hpsi(\chi) \in \Xi_{\Lz}^1$.
\eproof

\blemma\label{lem:mu_k_sums}
For $k \geq 3$, let $\mu_k$ denote the $k^\text{th}$ roots of unity and let $\zeta = \text{exp}(2\pi i / k)$. Suppose we have $a_1, \dots, a_n \in \mu_k \cup \{0\}$ such that $a_1 + \dots + a_n = n - 1 + \zeta$. Then there is a $j$ such that $a_j = \zeta$, and $a_i = 1$ for all $i \neq j$. 
\elemma

\bproof
If $k = 3$, then $\zeta$ is the only possible value with a positive imaginary part. Hence one of the $a_j$ equals $\zeta$, and then $a_i=1$ for the remaining $i \neq j$, since they have to sum to $n-1$.
For $k>3$, let $R:=\max \mathrm{Re}(\mu_k-\{1\})$, and let $f$ denote the number of $a_i$ that are not equal to $1$. Since $0 \leq R=\mathrm{Re}(\zeta)< 1$, we find that $n-1 + R =  n - 1 + \mathrm{Re}(\zeta) = \mathrm{Re}(a_1 + \dots + a_n) \leq n - f + f\text{Re}(\zeta) = n +f(R-1)$, so $R-1 \leq f(R-1)$, hence $f \leq 1$.
\eproof

\blemma
$\hpsi(\Xi_{\Kz}^2) = \Xi_{\Lz}^2$. 
\elemma

\bproof
For any character $\chi \in \Xi_{\Kz}^2$ we have $\chi^k \in \Xi_{\Kz}^1$. Thus, by Lemma~\ref{Delta1}, $\hpsi(\chi)^k \in \Xi_{\Lz}^1$. Hence for any $\mfq \in \cP_{\Lz, N}$ we have that $\hpsi(\chi)(\mfq) \in \mu_k \cup \{0\}$. As $\cX_N(\hpsi(\chi)) = \cX_N(\chi) = c_N - 1 + \zeta$, by the previous lemma there exists a single prime $\mfq_{\hpsi(\chi)}$ such that $\hpsi(\chi)(\mfq_{\hpsi(\chi)}) = \zeta$, while $\hpsi(\chi)(\mfq) = 1$ for all $\mfq \neq \mfq_{\hpsi(\chi)}$. Hence $\hpsi(\Xi_{\Kz}^2) \subseteq \Xi_{\Lz}^2$. By symmetry, we have equality.
\eproof

A character $\chi$ in $\Xi^2_{\Kz}$ has a special prime $\mfp_\chi$, and the corresponding character $\hpsi(\chi) \in \Xi^2_{\Lz}$ has a special prime $\mfq_{\hpsi(\chi)}$. We obtain an association of primes
\begin{align*}
\varphi_N: \cP_{\Kz, N} &\to \cP_{\Lz, N} \\
\mfp_\chi &\mapsto \mfq_{\hpsi(\chi)}.
\end{align*}
We show that this association is in fact a well-defined bijection with the required property. The first step is to show that every prime of $\cP_{\Kz, N}$ is associated to at least one prime of $\cP_{\Lz, N}$.

\blemma
For every $\mfp' \in \cP_{\Kz, N}$ there exists a character $\chi \in \Xi_{\Kz}^2$ such that $\mfp_\chi = \mfp'$. 
\elemma

\bproof
The Gr\"unwald-Wang Theorem (\cite{A-T}, Ch.\ X, Thm.\ 5) guarantees that there exists a character $\chi \in \Hab{\Kz}$ such that $\chi(\mfp') = \zeta$ and $\chi(\mfp) = 1$ for all primes $\mfp \neq \mfp'$ of norm $N$, because there exists a character of $\Gab_{\Kz_{\mfp'}}$ whose fixed field is the unique unramified extension of degree $k$ of $\Kz_{\mfp'}$. 
\eproof

\blemma
The association $\varphi_N$ is a well-defined bijection such that for every $\chi \in \Hab{\Kz}$ and $\mfp \in \cP_{\Kz, N}$ we have $\chi(\mfp) = \hpsi(\chi)(\varphi_N(\mfp))$.
\elemma

\bproof
Suppose we have $\chi, \chi' \in \Xi_{\Kz}^2$ such that $\mfp_{\chi} = \mfp_{\chi'}$ and $\varphi_N(\mfp_\chi) \neq \varphi_N(\mfp_{\chi'})$. We have $\cX_N(\chi \cdot \chi') = c_N - 1 + \zeta^2$, while $\cX_N(\hpsi(\chi \cdot \chi')) = \cX_N(\hpsi(\chi) \cdot \hpsi(\chi')) = c_N - 2 + 2 \zeta \neq c_N - 1 + \zeta^2$, which contradicts Lemma~\ref{lem:sum_equal}.
We conclude that $\varphi_N(\mfp_\chi) = \varphi_N(\mfp_{\chi'})$. Using $\hpsi^{-1}$ instead of $\hpsi$ provides a well-defined inverse $\varphi_N^{-1}$.

Take $\chi' \in \Xi_{\Kz}^2$ such that $\mfp_{\chi'} = \mfp$. We have $\cX_N(\chi \cdot \chi') = \cX_N(\chi) - \chi(\mfp) + \zeta \chi(\mfp) = \cX_N(\chi) - (1 - \zeta) \chi(\mfp)$. Similarly, $\cX_N(\hpsi(\chi \cdot \chi')) = \cX_N(\hpsi(\chi) \cdot \hpsi(\chi')) = \cX_N(\hpsi(\chi)) - (1 - \zeta)\hpsi(\chi)(\varphi_N(\mfp))$. As we have both $\cX_N(\chi) = \cX_N(\hpsi(\chi))$ and $\cX_N(\chi \cdot \chi') = \cX_N(\hpsi(\chi \cdot \chi'))$, we conclude that $(1 - \zeta) \chi(\mfp) = (1 - \zeta)\hpsi(\chi)(\varphi_N(\mfp))$ and consequently $\chi(\mfp) = \hpsi(\chi)(\varphi_N(\mfp))$.
\eproof

This completes the proof of Proposition~\ref{LIso-->reci}.

\section{Conditional reconstruction of global fields}

\bprop
Assume the equivalent statements \textup{(i)--(iii)} of Theorem~\ref{THM}. Then there exists an isomorphism $\Psi$ such that
\bgl \label{smalldiagram}
  \xymatrix@C=8mm{
  \Az_{\Kz,f}^* \cap \widehat{\cO}_{\Kz} \ar[d]_{\rec_{\Kz}} \ar[r]^{\Psi} & \Az_{\Lz,f}^* \cap \widehat{\cO}_{\Lz} \ar[d]^{\rec_{\Lz}} \\ 
  \Gab_{\Kz} \ar[r]^{\psi} & \Gab_{\Lz}
  }
\egl
commutes.
\eprop

\bproof
Define the homomorphism  $\widehat{\cO}_{\Kz}^* \times I_{\Kz} \to \Az_{\Kz,f}^* \cap \widehat{\cO}_{\Kz}$ by $(u,\mfm) \ma u \cdot s_{\Kz}(\mfm)$. It has a complete inverse $\Az_{\Kz,f}^* \cap \widehat{\cO}_{\Kz} \to \widehat{\cO}_{\Kz}^* \times I_{\Kz}$  given by $x \mapsto (x \cdot s_{\Kz}((x)_{\Kz})^{-1}, (x)_{\Kz})$.
We obtain an isomorphism $\Az_{\Kz,f}^* \cap \widehat{\cO}_{\Kz} \cong \widehat{\cO}_{\Kz}^* \times I_{\Kz}$ (and similarly for $\Lz$).

As seen in the proof of Proposition~\ref{reci-->N},
\[
\cO_{\mfp}^* \cong \rec_{\Kz}(\cO_{\mfp}^*) \stackrel{\psi}{\cong} \rec_{\Lz}(\cO_{\varphi(\mfp)}^*) \cong \cO_{\varphi(\mfp)}^*.
\]
We obtain an isomorphism $\Psi_{\mfp}: \cO_{\mfp}^* \cong \cO_{\varphi(\mfp)}^*$ that fits into the following commutative diagram:
\[ 
  \xymatrix@C=8mm{
  \cO_{\mfp}^* \ar[d]_{\rec_{\Kz}} \ar[r]^{\Psi_{\mfp}} & \cO_{\varphi(\mfp)}^* \ar[d]^{\rec_{\Lz}} \\ 
  \Gab_{\Kz} \ar[r]^{\psi} & \Gab_{\Lz}
  }
\]

As we have $\psi(\rec_{\Kz}(s_{\Kz}(\mfm))) = \rec_{\Lz}(s_{\Lz}(\varphi(\mfm)))$ for all $\mfm \in I_{\Kz}$ by assumption, the map $(\prod \Psi_{\mfp}) \times \varphi$ fits in the commutative diagram 
\[ 
  \xymatrix@C=8mm{
  \widehat{\cO}_{\Kz}^* \times I_{\Kz} \ar[d]_{\rec_{\Kz}} \ar[rr]^{(\prod \Psi_{\mfp}) \times \varphi} & & \widehat{\cO}_{\Lz}^* \times I_{\Lz} \ar[d]_{\rec_{\Lz}} \\ 
  \Gab_{\Kz} \ar[rr]^{\psi} & & \Gab_{\Lz}
  }
\]
With use of the aforementioned identifications $\Az_{\Kz,f}^* \cap \widehat{\cO}_{\Kz} \cong \widehat{\cO}_{\Kz}^* \times I_{\Kz}$ and $\Az_{\Lz,f}^* \cap \widehat{\cO}_{\Lz} \cong \widehat{\cO}_{\Lz}^* \times I_{\Lz}$ we obtain the desired isomorphism $\Psi$.
\eproof

\bcor
Diagram (\ref{smalldiagram}) can be extended to a commutative diagram 
\bgl \label{bigdiagram}
  \xymatrix@C=8mm{
  \Az_{\Kz,f}^*  \ar[d]_{\rec_{\Kz}} \ar[r]^{\Psi} & \Az_{\Lz,f}^* \ar[d]^{\rec_{\Lz}} \\ 
  \Gab_{\Kz} \ar[r]^{\psi} & \Gab_{\Lz}
  }
\egl
\bproof
The reciprocity map is already defined from  $\Az_{\Kz,f}^*$ to $\Gab_{\Kz}$, so this follows from (\ref{smalldiagram}) by passing to the group of fractions $\Az_{\Kz,f}^*$ of the monoid $\Az_{\Kz,f}^* \cap \widehat{\cO}_{\Kz}$.   
\eproof

\ecor

We now turn to the reconstruction of isomorphism of fields from the equivalent conditions in the main theorem \ref{THM}. For this, we first quote a result about conditions under which an isomorphism of multiplicative groups of fields can be extended to a field isomorphism. Let $$\Pi_{\mfp}: \: \Az_{\Kz,f}^\ast \onto \Kz_{\mfp}^*$$ be the canonical projection, and $\cO_{[\mfp]}$ the local ring (of $\Kz$) at the prime $\mfp$. 

\blemma[Uchida/Hoshi] \label{Hoshi}
An isomorphism $\Psi \colon \Kz^* \rightarrow \Lz^*$ of multiplicative groups of two global fields $\Kz$ and $\Lz$ is the restriction of an isomorphism of fields if and only if there exists a bijection $\varphi \colon \cP_{\Kz} \rightarrow \cP_{\Lz}$ such that for all  $\mfp \in \cP_{\Kz}$, both the following hold: \begin{itemize}
\item[\textup{(i)}] 
$ \Psi(1 + \mfp \cO_{[\mfp]}) = 1 + \varphi(\mfp) \cO_{[\varphi(\mfp)]}$ (as sets),  
\item[\textup{(ii)}] $v_{\varphi(\mfp)} \circ \Pi_{\varphi(\mfp)} \circ \Psi = v_{\mfp} \circ \Pi_{\mfp}.$
\end{itemize}

\elemma

\bproof
This follows immediately by results of Uchida for global function fields (\cite{U}, Lemmas 8-11; as explained in the introduction of \cite{Ho}) and by Hoshi for number fields (\cite{Ho}, Theorem D).
\eproof

\btheo \label{REC}
Assume that $\Psi$ in \textup{(\ref{bigdiagram})} above satisfies $\Psi(\Kz^*) = \Lz^*$. Then the extension of that isomorphism to a map $\Kz \to \Lz$ by setting $0 \mapsto 0$ is an isomorphism of fields. 
\etheo
\bproof
Fix $\mfp \in \cP_{\Kz}$.  By construction, there is an isomorphism $\Psi_{\mfp}: \: \Kz_{\mfp}^* \cong \Lz_{\varphi(\mfp)}^*$ such that $$\Pi_{\varphi(\mfp)} \circ \Psi = \Psi_{\mfp} \circ \Pi_{\mfp}.$$

The one-units of the complete local ring are simply the $n:=(N(\mfp)-1)$-th powers (\cite{Neu}, Proposition 5.7, Chapter II): \begin{equation} \label{student} 1 + \pi_\mfp \cO_{\mfp} = (\cO_{\mfp}^*)^{N(\mfp)-1}.\end{equation}

Since $\Psi_{\mfp}$ is multiplicative, we find $$\Psi_{\mfp}(1 + \mfp \cO_{\mfp}) = 1 + \varphi(\mfp) \cO_{\varphi(\mfp)}.$$ We conclude that for the local rings, we have
\bglnoz
  \Psi(1 + \mfp \cO_{[\mfp]}) &=& \Psi(\Kz^* \cap \Pi_{\mfp}^{-1}(1 + \mfp \cO_{\mfp})) \\ &=& \Psi(\Kz^*) \cap \Pi_{\varphi(\mfp)}^{-1}(1 + \varphi(\mfp) \cO_{\varphi(\mfp)})\\
  &=& 1 + \varphi(\mfp) \cO_{[\varphi(\mfp)]}.
\eglnoz
This proves condition (i) in Lemma \ref{Hoshi}. 
For condition (ii), observe that from the definition of $\Psi$ it follows that $\Psi(s_{\Kz}(\mfp)) = s_{\Lz}(\varphi(\mfp))$, and since $\Pi_\mfp(s_{\Kz}(\mfp))=\pi_\mfp$ by definition of a split (for a chosen uniformizer $\pi_\mfp$), the result follows. 

\eproof

\section{Reconstruction of global function fields} 
In function fields, we can immediately apply the results of the previous section: 

\bprop \label{PROP} If one of the equivalent conditions \textup{(i)--(iii)} of Theorem \ref{THM} holds for two global function fields $\Kz$ and $\Lz$, then they are isomorphic as fields. 
\eprop

\bproof
From the commutativity of diagram (\ref{bigdiagram}) we find that 
$ \Psi(\ker(\rec_{\Kz})) = \ker(\rec_{\Lz}). $
For a global function field $\Kz$ we have $\ker(\rec_{\Kz}) = \Kz^*$, so the result follows from Theorem \ref{REC}.
\eproof

\begin{remark} The equivalence of (v) in Theorem \ref{THM} and field isomorphism in global function fields was also shown in \cite{Cor} using dynamical systems, but referring to the unpublished \cite{CM} for a proof of the result in Section \ref{LtoR} of this paper. 
\end{remark}

\section{The number field case: an auxiliary result}\label{section:Mz}
In this section we prove the existence of certain Galois extensions of number fields with prescribed Galois groups; a result that we will use in the next section to prove the reconstruction of number fields from the consideration of specific induced representations.

\bprop\label{prop:wreath_ext}
Let $\Kz$ be a number field of degree $n$ contained in a Galois extension $\Nz$ of $\Qz$, and let $C$ be a finite cyclic group.  Denote $G = \normalfont{\Gal}(\Nz/\Qz)$ and $H = \normalfont{\Gal}(\Nz/\Kz)$ and let $C^n \rtimes G$ be the semidirect product of $C^n$ and $G$, where the action of $G$ on $C^n$ is by permuting coordinates the same way $G$ permutes the cosets $G/H$. By $C^n \rtimes H$ we denote the subgroup  of $C^n \rtimes G$ generated by $C^n$ and $H$. There exists a Galois extension $\Mz$ of $\mathbb{Q}$ containing $\Nz$ such that 
\[
\normalfont{\Gal}(\Mz/\Qz) = C^n \rtimes G, \; \normalfont{\Gal}(\Mz/\Kz) = C^n \rtimes H, \; \text{and } \normalfont{\Gal}(\Mz/\Nz) = C^n.
\]
\eprop

\begin{remark} The semidirect product $C^n \rtimes G$ is also known as the {\em wreath product} of $C$ and the group $G$ considered as a permutation group on $G/H$.
For any extension $\Lz$ of $\Kz$ with Galois group $C$, the Galois group of the Galois closure of $\Lz$ over $\Qz$ is a subgroup of this wreath product.
The proposition asserts that the wreath product itself (i.e., the maximal subgroup, which can be viewed as the `generic' case), actually occurs for some $\Lz$.

We give a self-contained proof, but the result also follows from Theorem IV.2.2 of \cite{MM}; or, for $C$ of order $3$ one can use the existence of a generic polynomial for $C$ and apply Proposition 13.8.2 in \cite{FJ}. 
\end{remark}

\begin{proof}[Proof of \ref{prop:wreath_ext}] 
Let $p \neq 2$ be a prime that is totally split in $\Nz$ and denote by $\mfp_1, \dots, \mfp_n$ the primes in $\Kz$ lying above $p$. There exists a Galois extension $\widetilde{\Kz}/\Kz$ with Galois group $C$ in which the prime $\mfp_1$ is inert, while $\mfp_2, \dots, \mfp_n$ are totally split (this follows, e.g., from the Gr\"unwald-Wang theorem). 

Let $X$ be the set of field homomorphisms from $\Kz$ to $\Nz$. Since $G$ acts on $\Nz$, we get an action of $G$ on $X$ by composition. This action is transitive and the 
stabilizer of the inclusion map $\iota\in X$ is $H$, so $X$ is isomorphic to $G/H$ as a $G$-set. For each $\sigma \in X$ we now consider
$\widetilde{\Nz}_{\sigma} := \widetilde{\Kz} \otimes_{\Kz, \sigma} \Nz$, where $\Nz$ is viewed as a $\Kz$-algebra through $\sigma\colon{\Kz} \to \Nz$.
The $C$-action on $\widetilde{\Kz}$ induces a $C$-action on $\widetilde{\Nz}_{\sigma}$ by $\Nz$-algebra automorphisms.
Setting $P_\sigma$ to be the set of primes of $\Nz$ that contain $\sigma(\mfp_1)$, we see that $\widetilde{\Nz}_{\sigma}$ is a
Galois extension of $\Nz$ with Galois group $C$ for which the primes in $P_\sigma$ are inert and all other primes of $\Nz$ over $p$ are totally split.

The $G$-action on the set of primes of $\Nz$ over $p$ is free and transitive, and $P_\iota$ consists of a single $H$-orbit: the primes over $\mfp_1$.
Since $P_{g\sigma}=gP_{\sigma}$ it follows that as $\sigma$ ranges over $X$ the sets $P_\sigma$ form a disjoint family. One deduces that
the fields $\widetilde{\Nz}_{\sigma}$ form a linearly disjoint family of $C$-extensions of $\Nz$ and that the tensor product
$$
\Mz=\bigotimes_{\sigma \in X} \widetilde{\Nz}_{\sigma}
$$
over $\Nz$ of all $\widetilde{\Nz}_{\sigma}$ with $\sigma\in X$ is a field which is Galois over
$\Nz$ with Galois group $C^n=\prod_{\sigma \in X} C$.

For $g\in G$ and $\sigma \in X$ there is a natural field isomorphism $\widetilde{g}_\sigma\colon \widetilde{\Nz}_{\sigma}\to \widetilde{\Nz}_{g\sigma}$ given by $x\otimes y\mapsto x \otimes gy$
that extends the map $\Nz\to \Nz$ given by $y\mapsto gy$.  Combining these maps for all $\sigma \in X$ we obtain an automorphism of the tensor product $\Mz$ that permutes the
factors of the tensor product by the $g$-action on $X$.  Thus, we have extended the $G$-action on $\Nz$ to a $G$-action on $\Mz$.  Since each
$\widetilde{g}_\sigma$ is $C$-equivariant, the subgroup of $\Aut(\Mz)$ generated by $G$ and $C^n$ is the semidirect product $C^n \rtimes G$.
As the cardinality of this group is the field degree of $\Mz$ over $\Qz$ we see that $\Mz$ is a Galois extension of $\Qz$ with Galois group $C^n \rtimes G$, and that
$\Kz$ is the invariant field of $C^n\rtimes H$.
\end{proof}
 
\section{Reconstruction of number fields} \label{section:NF}
Using the previous section we prove a stronger version of Proposition~\ref{PROP} for number fields.

\begin{theorem} \label{thm:Bart}
Let $\Kz$ be a number field and let $k \geq 3$. Then there exists a character $\chi \in \Hab{\Kz}$ of degree $k$ such that every number field $\Lz$ for which there is a character $\chi' \in \Hab{\Lz}$ with $L_{\Lz}(\chi') = L_{\Kz}(\chi)$ is isomorphic to $\Kz$.
\end{theorem}

We will use the following basic facts about Artin $L$-series of representations of the absolute Galois group
$G_{\Kz}:=\Gal(\widebar \Qz / \Kz)$ for a number field  $\Kz$ within a fixed algebraic closure $\widebar \Qz$ of $\Qz$.

\begin{lemma} \label{easy}\mbox{} 
\begin{enumerate}
\item[\textup{(a)}] For any two representations $\rho$ and $\rho'$ of $G_{\Qz}$, $L_{\Qz} (\rho) = L_{\Qz}(\rho')$ is equivalent to $\rho \cong \rho'$.
\item[\textup{(b)}] For $\chi \in \Hab{\Kz}$, we have $L_{\Kz}(\chi) = L_{\Qz}(\mathrm{Ind}^{G_\Qz}_{G_\Kz} (\chi))$. 
\item[\textup{(c)}] For any two number fields $\Kz$ and $\Lz$ within $\widebar \Qz$ and characters $\chi\in\Hab{\Kz}$ and $\chi'\in \Hab{\Lz}$ with $L_{\Kz} (\chi) = L_{\Lz}(\chi')$ we have an isomorphism of representations of $G_\Qz$
$$\mathrm{Ind}^{G_\Qz}_{G_\Kz} (\chi)\cong \mathrm{Ind}^{G_\Qz}_{G_\Lz} (\chi')$$
and the fixed fields $\Kz_\chi$ of $\chi$ and $\Lz_{\chi'}$ of $\chi'$ have the same normal closure over $\Qz$.
\end{enumerate}
\end{lemma}

\begin{proof} Fact (a) follows from Chebotarev's density theorem, comparing Euler factors. The basic fact (b) is due to Artin, see \cite{Neu} VII.10.4.(iv). The isomorphism in (c) follows
from (a) and (b). The last statement follows from the fact that the normal closure of $\Kz_\chi$ over $\Qz$ is the fixed field of the kernel of the representation $\mathrm{Ind}^{G_\Qz}_{G_\Kz} (\chi)$. \end{proof}

By a {\em monomial structure} of a representation $\rho$ of a group $G$ we mean a $G$-set ${\cL}$ consisting of $1$-dimensional subspaces of $\rho$ that
is $G$-stable (i.e. $gL\in\cL$ for all $g\in G$ and $L\in\cL$), and such that as a vector space we have $\rho=\bigoplus_{L\in \cL}  L$.
By choosing a single nonzero vector of $L$ for each $L\in\cL$ one obtains a basis of $\rho$ such that for every $g \in G$ the action of $g$ on $\rho$ is given by
a matrix with exactly one non-zero element in each row and in each column.
If $H$ is a subgroup of $G$ and $\chi$ a linear character of $H$, then the induced representation $\rho=\mathrm{Ind}^{{G}}_{{H}}(\chi)$ naturally produces a monomial structure
$\cL$ that is isomorphic to $G/H$ as a $G$-set. 

\begin{proof}[Proof of \ref{thm:Bart}] 
If $L_{\Lz}(\chi') = L_{\Kz}(\chi)$ for two characters $\chi \in \Hab{\Kz}, \chi' \in \Hab{\Lz}$, the lemma implies that $\mathrm{Ind}^{G_\Qz}_{G_\Kz} (\chi)$ has two monomial structures, one arising from $\chi$ and one from $\chi'$. We see that $\Kz$ and $\Lz$ are isomorphic as number fields if and only if these two monomial structures are isomorphic
as $G_\Qz$-sets.  In order to prove \ref{thm:Bart} it therefore suffices to choose $\chi$ in such a way that the representation 
$\mathrm{Ind}^{G_\Qz}_{G_\Kz} (\chi)$ only has a {\em single} monomial structure.

In order to find such a character $\chi$ we apply Proposition~\ref{prop:wreath_ext} where we let $C = \langle \zeta \rangle$ be the subgroup of $\Cz^\times$ generated by $\zeta = \text{exp}(2 \pi i / k)$. We let $n$, $G$, $H$, and $C^n\rtimes G$ be as in the proposition, and get an extension $\Mz$ of $\Kz$ within $\widebar \Qz$ with Galois group $\Gal(\Mz/\Qz)=C^n\rtimes G$.
We order the coordinates of $C^n$ in such a way that the action of $H$ on $C^n$ fixes the first coordinate, so the map
\[
\Gal(\Mz/\Kz)=C^n\rtimes H \to \Cz^\times: \:(a_1, \dots, a_n, h) \mapsto a_1
\]
is a group homomorphism, and extends to a character $\chi \in \Hab{\Kz}$.
The induced representation $\rho=\mathrm{Ind}^{G_\Qz}_{G_\Kz} (\chi)$ factors over $\Gal(\Mz/\Qz)=C^n\rtimes G$ and it comes with a monomial structure $\cL=\{L_1,\ldots,L_n\}$ such that each element $(a_1,\ldots,a_n)\in C^n$ acts on $L_i$ as scalar multiplication by $a_i$.  It follows that
$\cL$ is exactly the set of 1-dimensional $C^n$-submodules of $\rho$, the so-called character eigenspaces for the action of $C^n$ on $\rho$.

To finish the proof we will show that $\cL$ is the unique monomial structure on the representation $\rho$.
Suppose that $\cM$ is another monomial structure on $\rho$.
The trace of the element $c=(\zeta, 1, \ldots,1)\in C^n$ on $\rho$ is equal to $n-1+\zeta$. 
On the other hand, $c$ permutes the elements of $\cM$, so the trace of $c$ on $\rho$ is also equal to the sum of $k$-th roots of unity $\zeta_M \in \mu_k$ where $M$
ranges over those lines $M\in\cM$ with $cM=M$, and $\zeta_M$ is the scalar by which $c$ then acts on $M$.  Since $k\ge 3$ and $\cL$ and $\cM$ have the same number of elements, Lemma~\ref{lem:mu_k_sums} implies that $cM=M$
for all $M\in \cM$. It follows that $c$ acts trivially on the set $\cM$. Since the $G$-conjugates of $c$ generate $C^n$ we deduce that $C^n$ acts trivially on the set $\cM$.
Thus, $\cM$ consists of $1$-dimensional $C^n$-submodules of $\rho$. This implies that $\cM\subset \cL$, so $\cM=\cL$ for cardinality reasons.
\eproof

\begin{remark} Not every representation has a unique monomial structure: consider the isometry group of a square, the dihedral group $D_4$ of order $8$, with its standard 2-dimensional representation.
It has two distinct monomial structures (consisting of the axes and the diagonals) and these are not isomorphic as $D_4$-sets.

By realizing $D_4$ as a Galois group over $\Qz$  this gives rise to quadratic fields $\Kz$ and  and $\Lz$  with quadratic characters $\chi\in\Hab{\Kz}$ and $\chi'\in \Hab{\Lz}$
that satisfy $L_{\Kz} (\chi) = L_{\Lz}(\chi')$ while  $\Kz\not\cong\Lz$. This shows that the method of proof of the theorem fails without
the assumption that $k\geq 3$.

Concretely, $\mathrm{Gal}(\Qz(\sqrt[4]{2}, i)/\Qz) \cong D_4$, and we find  $L_{\Kz}(\chi)=L_{\Lz}(\chi')$ for $\Kz=\Qz(\sqrt{2})$, $\Lz=\Qz(i\sqrt{2})$ and $\chi$ and $\chi'$ uniquely determined by $\Kz_{\chi} = \Qz(\sqrt[4]{2})$ and $\Lz_{\chi'}=\Qz(i \sqrt{2}, (1+i)\sqrt[4]{2})$. 
\end{remark}

\section{Comparison of different methods of proof}
There is an interesting ``incompatibility of proof techniques'' between the case of global function fields and that of number fields. Namely, the approach of the proof of Proposition~\ref{PROP} for function fields does not transfer in an obvious way to number fields. Indeed, for a number field $\Kz$, $\ker(\rec_{\Kz}) = \Kz^* \cdot \overline{\cO_{\Kz,+}^*}$, where $\overline{\cO_{\Kz,+}^*}$ is the closure of the totally positive units of the field in the ideles; this follows from the description of the connected component of the idele class group by Artin, cf.\ Chapter IX in \cite{A-T}. Hence the method of proof of \ref{PROP} transferred literally to number fields yields the weaker conclusion that $$\Psi(\Kz^* \cdot \overline{\cO_{\Kz,+}^*}) = \Lz^* \cdot \overline{\cO_{\Lz,+}^*}.$$ It is unclear to us whether one can deduce that $\Psi(\Kz^*) = \Lz^*$ from the conditions in Theorem \ref{THM}. The issue is similar to the one raised in \cite{H}, 3.3.2. 

On the other hand, it is not possible to copy the proof of Theorem~\ref{thm:Bart} for function fields, as this would force fixing a rational subfield $\mathbb{F}_q(t)$ inside both $\Kz$ and $\Lz$  (that plays the role of $\mathbb{Q}$ in the number field proof), for which there are infinitely many, non-canonical, choices. However, Theorem~\ref{thm:Bart} does hold in the relative setting of separable geometric extensions of a \emph{fixed} rational function field of characteristic not equal to $2$, compare \cite{S}. It is unclear to us whether the analogue of Theorem \ref{thm:Bart} holds for a global function field without fixing a rational subfield. It does seem that $L$-series of global function fields, as polynomials in $q^{-s}$, contain less arithmetical information than their number field cousins (compare \cite{Gross}).

\end{document}